\documentclass{amsart}

\usepackage{amsmath,amssymb}	

\usepackage{graphicx}		  

\usepackage{amsthm}
\usepackage{ascmac}
\usepackage{cases}
\usepackage{multicol}
\usepackage[all]{xy}
\usepackage{nccmath}
\usepackage{amscd} 
\usepackage{float}

\newtheorem{theorem}{Theorem}
\newtheorem{lemma}{Lemma}
\newtheorem{proposition}{Proposition}
\newtheorem{corollary}{Corollary}
\newtheorem{definition}{Definition}

\newtheorem{example}{Example}

\newcommand{\id}{\mathrm{id}}
\newcommand{\aut}{\mathrm{Aut}}
\newcommand{\inn}{\mathrm{Inn}}
\newcommand{\conj}{\mathrm{Conj}}
\newcommand{\orb}{\mathrm{Orb}}
\newcommand{\pqnd}{\mathcal{P}_{\mathsf{Qnd}}}
\newcommand{\p}{\mathcal{P}}
\newcommand{\depth}{\mathrm{depth}}
\newcommand{\simmcq}{\overset{\mathsf{MCQ}}{\sim}}
\newcommand{\ranglemcq}{\rangle_\mathsf{MCQ}}

\title{On connected component decompositions of quandles}

\author[Y. Iijima]{Yusuke Iijima}
\address[Y. Iijima]
{Institute of Mathematics, University of Tsukuba,\\
1-1-1 Tennoudai, Tsukuba, Ibaraki 305-8571, Japan.}
\email{yu-suke-i@math.tsukuba.ac.jp}

\author[T. Murao]{Tomo Murao}
\address[T. Murao]
{Institute of Mathematics, University of Tsukuba,\\
1-1-1 Tennoudai, Tsukuba, Ibaraki 305-8571, Japan.}
\email{t-murao@math.tsukuba.ac.jp}


\begin{document}

\begin{abstract}

We give a formula of the connected component decomposition of 
the Alexander quandle: 
$\mathbb{Z}[t^{\pm1}]/(f_1(t),\ldots, f_k(t))=\bigsqcup^{a-1}_{i=0}\orb(i)$,
where $a=\gcd (f_1(1),\ldots, f_k(1))$. 
We show that the connected component $\orb(i)$ is isomorphic to $\mathbb{Z}[t^{\pm1}]/J$
with an explicit ideal $J$.
By using this, we see how a quandle is decomposed into connected components
for some Alexander quandles.
We introduce a decomposition of a quandle into the disjoint union of maximal connected subquandles. 
In some cases, this decomposition is obtained by iterating a connected component decomposition. 
We also discuss
the maximal connected sub-multiple conjugation quandle decomposition.
\end{abstract}

\maketitle

\section{Introduction}

A quandle \cite{Joyce82,Matveev82} is an algebraic structure 
whose axioms are derived from the Reidemeister moves 
on oriented link diagrams.
An inner automorphism group of a quandle has 
an action to the quandle naturally.
We call an orbit of the quandle 
by the action its connected component, 
which is a subquandle.
A quandle is said to be connected 
if the action is transitive.
It is known that 
all connected quandles of prime square order are 
Alexander quandles \cite{Grana04}.

Any connected component of an Alexander quandle $M$ 
is isomorphic to $(1-t)M$.
S. Nelson \cite{Nelson03} proved that 
two finite Alexander quandles $M$ and $N$ of the same cardinality are isomorphic  
if and only if 
$(1-t)M$ and $(1-t)N$ are isomorphic as modules, 
and showed connectivity of some Alexander quandles.
The numbers of Alexander quandles and connected ones 
are listed up to order 16 in \cite{Nelson03,NG08}.
In this paper, 
for any $f_1(t),\ldots, f_k(t) \in \mathbb{Z}[t^{\pm1}]$, 
we show that the connected component decomposition of 
the Alexander quandle $\mathbb{Z}[t^{\pm1}]/(f_1(t),\ldots, f_k(t))$
is $\bigsqcup^{a-1}_{i=0}\orb(i)$ 
and that $\orb(i)$ is isomorphic to $\mathbb{Z}[t^{\pm1}]/J$
with an explicit ideal $J$, 
where $a=\gcd (f_1(1),\ldots, f_k(1))$.

A connected subquandle has played an important roll in colorings of a knot diagram.
However 
a connected component of a quandle is not a connected quandle in general.
In this paper, 
we introduce a decomposition of a quandle into 
the disjoint union of maximal connected subquandles, 
and 
show that 
it is obtained by iterating a connected component decomposition 
when the quandle is finite, 
where we note that 
the decomposition of a finite quandle 
obtained by iterating a connected component decomposition 
was introduced in \cite{EGTY08,NW06}.
We also give examples of the decompositions of some quandles.
For example, 
we concretely determine the decompositions of 
the Alexander quandle $\mathbb{Z}[t^{\pm 1}]/(n_0,t+a)$ 
for any $n_0 \in \mathbb{Z}_{>0}$ and $a\in \mathbb{Z}$ 
and the dihedral quandle $R_m$ 
for any $m \in \mathbb{Z}_{\geq 0}$.

We also discuss 
the similar decomposition 
of a multiple conjugation quandle, 
which is an algebraic structure 
whose axioms are derived from the Reidemeister moves 
on diagrams of spatial trivalent graphs and handlebody-links.

This paper is organized as follows.
In Section 2, 
we recall the definition of a quandle 
and its connected components.
In Section 3, 
we determine the connected component decomposition of 
the Alexander quandle $\mathbb{Z}[t^{\pm1}]/(f_1(t),\ldots, f_k(t))$.
In Section 4, 
we introduce a decomposition of a quandle into 
the disjoint union of maximal connected subquandles 
and show that 
it is obtained by iterating a connected component decomposition 
when the quandle is finite.
In Section 5, 
we give examples 
of the maximal connected subquandle decompositions 
of some quandles.
In Section 6, 
we recall the definition of 
a multiple conjugation quandle 
and introduce its properties we use in Section 7.
In Section 7, 
we discuss 
the similar decomposition 
of a multiple conjugation quandle.

\section{A quandle}

A \emph{quandle} $X$ is a non-empty set 
with a binary operation $* : X \times X \to X$ 
satisfying the following axioms.
\begin{itemize}
\item
For any $a \in X$, 
$a * a = a$.
\item
For any $a \in X$, 
the map $S_a : X \to X$ defined by $S_a(x)=x*a$ is a bijection.
\item
For any $a,b,c \in X$, 
$(a*b)*c=(a*c)*(b*c)$.
\end{itemize}

We give some examples of quandles.
Let $G$ be a group. 
We define a binary operation $*:G\times G\to G$ 
by $a*b=b^{-1}ab$ for any $a,b\in G$.
Then $G$ is a quandle, 
which is called a \emph{conjugation quandle} 
and denoted by $\mathrm{Conj}(G)$. 
The second example is obtained from a $\mathbb{Z}[t^{\pm 1}]$-module $M$.
We define a binary operation $*:M\times M\to M$ 
by $a*b=ta+(1-t)b$ for any $a,b\in M$.
Then $M$ is a quandle, 
which is called an \emph{Alexander quandle}.
The third example is a \emph{dihedral quandle} $R_m := \mathbb{Z}/m\mathbb{Z}$ 
for any $m\in\mathbb{Z}_{\geq 0}$.
We define a binary operation $*:R_m\times R_m\to R_m$ 
by $a*b=2b-a$ for any $a,b\in R_m$.
Then $R_m$ is a quandle.

Let $(X,*)$ be a quandle.
In this paper, 
we write $a_1*a_2*\cdots *a_n$ 
for $(\cdots ((a_1*a_2)*a_3)*\cdots *a_n)$ simply. 
For any $a\in X$ and $n\in \mathbb{Z}_{\geq 0}$, 
$S^{n}_a:X\to X$ and $S^{-n}_a:X\to X$
are $n$-th functional powers of $S_a$ and $S^{-1}_a$ respectively.
For any $a,b\in X$ and $n \in \mathbb{Z}$, 
we define a binary operation $*^n : X\times X\to X$ 
by $a*^nb=S^n_b(a)$.
Then $(X,*^{n})$ is also a quandle.
We define the \emph{type} of $X$ 
by the minimal number of $n$ 
satisfying $a*^nb=a$ for any $a,b\in X$.

Let $(X,*_X)$ and $(Y,*_Y)$ be quandles. 
A \emph{homomorphism} $\phi : X \to Y$ is a map from $X$ to $Y$ 
satisfying $\phi(x*_X y)=\phi(x)*_Y \phi(y)$ for any $x,y\in X$. 
We call a bijective homomorphism an \emph{isomorphism}. 
$X$ and $Y$ are \emph{isomorphic}, 
denoted $X\cong Y$, 
if there exists an isomorphism from $X$ to $Y$.
We call an isomorphism from $X$ to $X$ an \emph{automorphism} of $X$. 
For any $a\in X$ and $n\in \mathbb{Z}$,
the map $S_a^n : X \to X$ is an automorphism of $X$.

%

Let $(X,*)$ be a quandle. 
A non-empty subset $Y$ of $X$ is called a \emph{subquandle} of $X$ 
if $Y$ itself is a quandle under $*$. 
For any subset $Y$ of $X$, 
$Y$ is a subquandle of $X$ 
if and only if 
$a*b,a*^{-1}b\in Y$ for any $a,b\in Y$.

%

For any subset $A$ of $X$,
the minimal subquandle of $X$ including $A$, 
denoted by $\langle A \rangle$, 
is called the subquandle generated by $A$, 
that is, 
\begin{eqnarray*}
\langle A \rangle =\{ a*^{k_1}x_1*^{k_2}\dots *^{k_n}x_n \in X\,|\, a,x_1,\dots,x_n\in A ,\  k_1,k_2,\dots ,k_n\in \mathbb{Z} \}. 
\end{eqnarray*}

%

Let $X$ be a quandle.
All automorphisms of $X$ form a group 
under composition of morphisms: $f \cdot g :=g \circ f$.
This group is called the \emph{automorphism group} of $X$ 
and denoted by $\aut (X)$.
For a subset $A$ of $X$, 
we denote by  $\inn (A)$ 
a subgroup of $\aut (X)$ generated by $\{ S_a \mid a \in A \}$.
In particular, $\inn (X)$ is called the \emph{inner automorphism group} of $X$. 
For any $a \in X$ and $g \in \inn (A)$, 
we define an action of $\inn (A)$ on $X$ 
by $a \cdot g = g(a)$.
We say that $X$ is a \emph{connected quandle} 
when the action is transitive.
In general, 
an orbit of $X$ by the action is called 
a \emph{connected component} of $X$ 
or 
an orbit of $X$ simply, 
and $X=\bigsqcup_{i \in I} X_i$ is called 
the \emph{connected component decomposition} of $X$ 
when $X_i$ is a connected component of $X$ for any $i \in I$.
In general, a connected component of $X$ is a subquandle of $X$.
We denote by $\orb_X(a)$ or $\orb(a)$ 
the orbit of $X$ containing $a$.

\begin{example}\label{conjugation}
For any group $G$,
a connected component of $\conj (G)$ 
coincides with one of conjugacy classes of $G$.
\end{example}

In the following,
we give a well-known fact with a proof.

\begin{lemma}\label{conn. comp. iso.}
Let $M$ be an Alexander quandle.
Then any connected component of $M$ 
is isomorphic to $(1-t)M$.
\end{lemma}

\begin{proof}
Since $0*x=(1-t)x$ for any $x\in M$,
it follows that $(1-t)M \subset \orb (0)$.
On the other hand, 
for any $z\in \orb (0)$, 
there exist $y_1,\dots ,y_n \in M$ 
such that $z=0*y_1*y_2*\cdots *y_n=(1-t)y_1*y_2*\cdots *y_n$.
Since for any $x,y \in M$, 
$(1-t)x*y=t(1-t)x+(1-t)y=(1-t)(tx+y)$, 
we have $z\in (1-t)M$, 
that is, $(1-t)M \supset \orb (0)$.
Therefore $(1-t)M= \orb (0)$.
Next, we define the map $\phi _a: \orb (0) \to \orb (a)$ 
by $\phi _a(x)=x+a$ for any $a \in M$.
For any $(1-t)x\in \orb(0)$ and $a\in M$,
$\phi _a((1-t)x)=(1-t)x+a=ta+(1-t)(x+a) = a*(x+a) \in \orb (a)$.
Hence $\phi _a$ is well-defined.
For any $x,y\in \orb (0)$ and $a\in M$,
$\phi _a(x*y)=(tx+(1-t)y)+a=t(x+a)+(1-t)(y+a)=\phi _a (x)*\phi _a(y)$.
Hence $\phi _a$ is a homomorphism.
Similarly, the map $\psi _a: \orb (a) \to \orb (0)$ 
defined by $\psi _a(x)=x-a$ is a homomorphism 
for any $a\in M$.
Since $\phi _a\circ \psi _a=\id _{\orb (a)}$ 
and $\psi _a\circ \phi _a=\id _{\orb (0)}$, 
$\phi _a$ is an isomorphism.
Therefore $\orb (0)$ and $\orb (a)$ are isomorphic for any $a \in M$.
\end{proof}

\section{The connected component decomposition of an Alexander quandle}

In this section, we show that the connected component decomposition of 
the Alexander quandle $\mathbb{Z}[t^{\pm 1}]/(f_1(t),f_2(t),\dots ,f_k(t))$ 
is $\bigsqcup^{a-1}_{i=0}\orb(i)$, where $(f_1(t),f_2(t), \dots ,f_k(t))$ 
is an ideal of $\mathbb{Z}[t^{\pm 1}]$ 
generated by Laurent polynomials $f_1(t),f_2(t),\dots ,f_k(t) \in \mathbb{Z}[t^{\pm 1}]$, 
and $a=\gcd(f_1(1),f_2(1),\dots ,f_k(1))$.
Furthermore, we determine the form of all connected components of 
$\mathbb{Z}[t^{\pm 1}]/(f_1(t),f_2(t), \dots ,f_k(t))$.

For any $\alpha (t) \in \mathbb{Z}[t^{\pm 1}]$,
we define $C_{\alpha (t)}$ by
\begin{eqnarray*}
C_{\alpha (t)}=\{ \alpha(t) + a_1f_1(t)+a_2f_2(t)+\dots +a_kf_k(t) \, |\, a_1,a_2, \dots ,a_k\in \mathbb{Z}[t^{\pm 1}] \} \subset \mathbb{Z}[t^{\pm 1}].
\end{eqnarray*}
Then for any $[\alpha(t)]\in \mathbb{Z}[t^{\pm 1}]/(f_1(t),f_2(t),\dots ,f_k(t))$,
$C_{[\alpha(t)]}:=C_{\alpha(t)}$ is well-defined.
In this paper,
we often write $\alpha (t)$ for $[\alpha(t)]\in \mathbb{Z}[t^{\pm 1}]/(f_1(t),f_2(t), \dots ,f_k(t))$ simply.
For any $D\subset \mathbb{Z}[t^{\pm 1}]$, 
we define $D(1)$ 
by $D(1)=\{ g(1) \,|\, g(t)\in D \}\subset \mathbb{Z}$.
It is easy to see that
\begin{eqnarray*}
C_{[\alpha(t)]}(1) = \{ \alpha(1) + a_1f_1(1)+a_2f_2(1)+\dots +a_kf_k(1) \, |\, a_1,a_2, \dots ,a_k\in \mathbb{Z} \}
\end{eqnarray*} 
for any $[\alpha (t)] \in \mathbb{Z}[t^{\pm 1}]/(f_1(t),f_2(t),\dots ,f_k(t))$.

\begin{lemma}\label{qnd. coeff.}
For any elements $[\alpha (t)]$ and $[\beta(t)]$ 
of the Alexander quandle $\mathbb{Z}[t^{\pm 1}]/(f_1(t),f_2(t),\dots ,f_k(t))$,
it follows that 
$C_{[\alpha(t)]*[\beta(t)]}(1)=C_{[\alpha(t)]}(1)$.
\end{lemma}

\begin{proof}
Since $[\alpha(t)]*[\beta(t)]=[t\alpha(t)+(1-t)\beta(t)]$,
we have 
\begin{eqnarray*}
C_{[\alpha(t)]*[\beta(t)]}(1)
&=&\{ 1\cdot\alpha(1)+(1-1)\beta(1) + a_1f_1(1)+\dots +a_kf_k(1) \, |\, a_1,a_2, \dots ,a_k\in \mathbb{Z} \} \\
&=&\{ \alpha(1)+0\cdot \beta(1) + a_1f_1(1)+\dots +a_kf_k(1) \, |\, a_1,a_2, \dots ,a_k\in \mathbb{Z} \} \\
&=& C_{[\alpha (t)]}(1). 
\end{eqnarray*}
\end{proof}

\begin{lemma}\label{orb. coeff.}
For any elements $[\alpha (t)]$ and $[\beta(t)]$ 
of the Alexander quandle $\mathbb{Z}[t^{\pm 1}]/(f_1(t),f_2(t),\dots ,f_k(t))$,
it follows that 
$C_{[\alpha (t)]}(1)=C_{[\beta (t)]}(1)$ if and only if $\orb ([\alpha(t)])=\orb ([\beta(t)])$.
\end{lemma}

\begin{proof}
Suppose that $C_{[\alpha (t)]}(1)=C_{[\beta (t)]}(1)$.
Then we have 
\begin{eqnarray*}
&&\{ \alpha(1) +l_1f_1(1)+l_2f_2(1)+\dots +l_kf_k(1)\, |\, l_1,l_2,\dots ,l_k\in \mathbb{Z} \}\\
&&=C_{[\alpha (t)]}(1)\\
&&=C_{[\beta (t)]}(1)\\
&&=\{ \beta (1) +l_1f_1(1)+l_2f_2(1)+\dots +l_kf_k(1)\, |\, l_1,l_2,\dots ,l_k\in \mathbb{Z} \}.
\end{eqnarray*}
Hence there exist $l_1,l_2,\dots ,l_k\in \mathbb{Z}$ such that 
$\alpha(1)=\beta(1)+l_1f_1(1)+l_2f_2(1)+\dots +l_kf_k(1)$.
We put $\widetilde{\alpha}(t)$ and $\widetilde{\beta}(t)$ 
by $\alpha(t)=(1-t)\widetilde{\alpha}(t)+\alpha(1)$ 
and $\beta(t)=(1-t)\widetilde{\beta}(t)+\beta(1)$ respectively.
We also put $\widetilde{f_i}(t)$ 
by $f_i(t)=(1-t)\widetilde{f_i}(t)+f_i(1)$ for any $i=1,2,\dots ,k$ 
and $\gamma(t):=\widetilde{\beta}(t)+\beta(1)-t\widetilde{\alpha}(t)+\sum^{k}_{i=1}l_it\widetilde{f_i}(t)$.
Since $[\sum^{k}_{i=1}l_itf_i(t)]=[0]$ 
in $\mathbb{Z}[t^{\pm 1}]/(f_1(t),f_2(t),\dots ,f_k(t))$, 
we have
\begin{eqnarray*}
&& [(1-t)\gamma(t)]\\
&&=[(1-t)(\widetilde{\beta}(t)+\beta(1)-t\widetilde{\alpha}(t)+\sum^{k}_{i=1}l_it\widetilde{f_i}(t))]\\
&&=[(1-t)\widetilde{\beta}(t)+(1-t)\beta(1)-t(1-t)\widetilde{\alpha}(t)+\sum^{k}_{i=1}l_it(1-t)\widetilde{f_i}(t)]\\
&&=[(1-t)\widetilde{\beta}(t)+(1-t)\beta(1)-t(1-t)\widetilde{\alpha}(t)+\sum^{k}_{i=1}l_it(1-t)\widetilde{f_i}(t)-\sum^{k}_{i=1}l_itf_i(t)]\\
&&=[(1-t)\widetilde{\beta}(t)+\beta(1)-t\beta(1)-t(1-t)\widetilde{\alpha}(t)+\sum^{k}_{i=1}l_it((1-t)\widetilde{f_i}(t)-f_i(t))]\\
&&=[\beta(t)-t(1-t)\widetilde{\alpha}(t)-t\beta(1)+\sum^{k}_{i=1}l_it(-f_i(1))]\\
&&=[\beta(t)-t(1-t)\widetilde{\alpha}(t)-t\alpha(1)]\\
&&=[\beta(t)-t\alpha(t)].
\end{eqnarray*}
Hence $[\alpha(t)]*[\gamma(t)]=[t \alpha (t)+(1-t) \gamma(t)]=[\beta(t)]$, 
which implies $\orb([\alpha(t)])=\orb([\beta(t)])$.
Next, suppose that $\orb([\alpha(t)])=\orb([\beta(t)])$.
There exist $[\gamma_1(t)],[\gamma_2(t)],\dots ,[\gamma_l(t)]\in\mathbb{Z}[t^{\pm 1}]/(f_1(t),f_2(t),\dots ,f_k(t))$ 
and $\epsilon_1,\epsilon_2,\dots ,\epsilon_l\in\mathbb{Z}$ such that 
$[\alpha(t)]*^{\epsilon_1}[\gamma_1(t)]*^{\epsilon_2}\dots *^{\epsilon_l}[\gamma_l(t)]=[\beta(t)]$.
By Lemma \ref{qnd. coeff.}, we have $C_{[\alpha (t)]}(1)=C_{[\beta (t)]}(1)$.
\end{proof}



Then we have the following lemma.

\begin{lemma}\label{orb. gcd}
For any elements $[\alpha (t)]$ and $[\beta(t)]$ 
of the Alexander quandle $\mathbb{Z}[t^{\pm 1}]/(f_1(t),f_2(t),\dots ,f_k(t))$,
it follows that 
$\orb ([\alpha(t)])=\orb ([\beta(t)])$ 
if and only if 
$\alpha (1)\equiv\beta (1)\mod a$, 
where $a=\gcd(f_1(1),f_2(1),\dots ,f_k(1))$.
\end{lemma}

\begin{proof}
Suppose that $\orb ([\alpha(t)])=\orb ([\beta(t)])$.
By Lemma \ref{orb. coeff.}, 
$C_{[\alpha (t)]}(1)=C_{[\beta (t)]}(1)$.
Since $f_1(1)\mathbb{Z}+f_2(1)\mathbb{Z}+\cdots +f_k(1)\mathbb{Z}=a\mathbb{Z}$,
we have 
\begin{eqnarray*}
\{ \alpha(1) +la\, |\, l \in \mathbb{Z} \}
&&= \{ \alpha(1) +l_1f_1(1)+l_2f_2(1)+\dots +l_kf_k(1)\, |\, l_1,l_2,\dots ,l_k\in \mathbb{Z} \}\\
&&= C_{[\alpha (t)]}(1)\\
&&= C_{[\beta (t)]}(1)\\
&&= \{ \beta (1) +l_1f_1(1)+l_2f_2(1)+\dots +l_kf_k(1)\, |\, l_1,l_2,\dots ,l_k\in \mathbb{Z} \}\\
&&=\{ \beta (1) +la\, |\, l \in \mathbb{Z} \}.
\end{eqnarray*}
Hence we obtain $\alpha (1) \equiv \beta (1)\mod a$.
On the other hand, 
suppose that $\alpha (1)\equiv\beta (1)\mod a$.
Since $f_1(1)\mathbb{Z}+f_2(1)\mathbb{Z}+\cdots +f_k(1)\mathbb{Z}=a\mathbb{Z}$,
we have 
\begin{eqnarray*}
C_{[\alpha (t)]}(1)
&&= \{ \alpha(1) +l_1f_1(1)+l_2f_2(1)+\dots +l_kf_k(1)\, |\, l_1,l_2,\dots ,l_k\in \mathbb{Z} \}\\
&&= \{ \alpha(1) +la\, |\, l \in \mathbb{Z} \}\\
&&=\{ \beta (1) +la\, |\, l \in \mathbb{Z} \}\\
&&= \{ \beta (1) +l_1f_1(1)+l_2f_2(1)+\dots +l_kf_k(1)\, |\, l_1,l_2,\dots ,l_k\in \mathbb{Z} \}\\
&&= C_{[\beta (t)]}(1).
\end{eqnarray*}
By Lemma \ref{orb. coeff.}, 
we obtain $\orb ([\alpha(t)])=\orb ([\beta(t)])$.
\end{proof}

Then we have the following theorem.

\begin{theorem}\label{Alex. ccd}
Let $M$ be the Alexander quandle 
$\mathbb{Z}[t^{\pm 1}]/(f_1(t),f_2(t),\dots ,f_k(t))$ 
and let $a=\gcd(f_1(1),f_2(1), \dots ,f_k(1))$.
Then the following hold. 
\begin{enumerate}
\item 
The connected component decomposition of $M$ is given by 
\[ M=\bigsqcup^{a-1}_{i=0}\orb (i), \]
where 
\begin{align*}
\orb(i)
& =\{ [g(t)] \mid g(t)\in \mathbb{Z}[t^{\pm 1}],~ g(1)\equiv i\mod a \}\\
& =\{ [i+(1-t)g(t)+aj] \mid g(t)\in\mathbb{Z}[t^{\pm 1}],~ j\in \mathbb{Z} \}.
\end{align*}

\item 
For any $j=0,1,\dots ,a-1$, 
it follows that 
\[ \orb(j) \cong  \mathbb{Z}[t^{\pm 1}]/((f_1(t),f_2(t),\dots ,f_k(t))+I), \] 
where we define $\widetilde{f_i}(t)$ 
by $f_i(t)=(1-t)\widetilde{f_i}(t)+f_i(1)$ 
for any $i=1,2,\dots ,k$, 
and 
$I=\{ \sum_{i=1}^k a_i\widetilde{f_i}(t) \mid 
a_i \in \mathbb{Z}[t^{\pm 1}],~ \sum_{i=1}^k a_if_i(1)=0 \}$.
\end{enumerate}
\end{theorem}

\begin{proof}
\begin{enumerate}
\item 
By Lemma \ref{orb. gcd}, 
$M=\bigsqcup^{a-1}_{i=0}\orb (i)$ is the connected component decomposition of $M$, and we have 
$\orb(i)=\{ [g(t)] \mid g(t)\in \mathbb{Z}[t^{\pm 1}],~ g(1)\equiv i\mod a \}$ immediately. 
There exist $\widetilde{g}(t)\in\mathbb{Z}[t^{\pm 1}]$ and $j\in \mathbb{Z}$ such that $g(t)=g(1)+(1-t)\widetilde{g}(t)=i+(1-t)\widetilde{g}(t)+aj$
if and only if $g(1)\equiv i\mod a$. 
Hence we have 
$\orb(i)=\{ [i+(1-t)g(t)+aj] \mid g(t)\in\mathbb{Z}[t^{\pm 1}],~ j\in \mathbb{Z} \}$. 

\item
Let $J=(f_1(t),f_2(t),\dots ,f_k(t))$.
We define the $\mathbb{Z}[t^{\pm 1}]$-homomorphism 
$\phi : \mathbb{Z}[t^{\pm 1}]\to (1-t)M$ 
by $\phi (x)=(1-t)x+J$.
It is clear that $J\subset \ker (\phi)$.
For any 
$\sum _{i=1}^{k}a_i\widetilde{f_i}(t)\in I$, 
we have 
\begin{eqnarray*}
\phi (\sum _{i=1}^{k}a_i\widetilde{f_i}(t))
&=&(1-t)\sum _{i=1}^{k}a_i\widetilde{f_i}(t)+J\\
&=&\sum _{i=1}^{k}a_i(1-t)\widetilde{f_i}(t)+\sum _{i=1}^{k}a_if_i(1)+J\\
&=&\sum _{i=1}^{k}a_i((1-t)\widetilde{f_i}(t)+f_i(1))+J\\
&=&\sum _{i=1}^{k}a_if_i(t)+J,
\end{eqnarray*}
which implies that $I\subset \ker (\phi)$.
Hence we obtain $J+I\subset \ker (\phi)$.
On the other hand, let $g(t)\in \ker (\phi)$.
Since $\phi(g(t))=(1-t)g(t)+J=J$,
there exist $h_1(t),h_2(t),\dots ,h_k(t)\in \mathbb{Z}[t^{\pm 1}]$
such that $(1-t)g(t)=\sum _{i=1}^k h_i(t)f_i(t)$.
When we put $t=1$, 
we have $0=\sum _{i=1}^k h_i(1)f_i(1)$.
For any $i=1,2,\dots ,k$, 
we define $\widetilde{h_i}(t)$ by $h_i(t)=(1-t)\widetilde{h_i}(t)+h_i(1)$.
Since $\sum_{i=1}^k h_i(1)f_i(1)=0$,
we have 
\begin{eqnarray*}
(1-t)g(t)
&=&\sum_{i=1}^k h_i(t)f_i(t)\\
&=&\sum_{i=1}^k ((1-t)\widetilde{h_i}(t)f_i(t)+(1-t)h_i(1)\widetilde{f_i}(t)+h_i(1)f_i(1))\\
&=&(1-t)\sum_{i=1}^k \widetilde{h_i}(t)f_i(t)+(1-t)\sum_{i=1}^kh_i(1)\widetilde{f_i}(t)
+\sum_{i=1}^k h_i(1)f_i(1)\\
&=&(1-t)\sum_{i=1}^k \widetilde{h_i}(t)f_i(t)+(1-t)\sum_{i=1}^kh_i(1)\widetilde{f_i}(t).
\end{eqnarray*}
Hence we have 
$g(t)=\sum_{i=1}^k \widetilde{h_i}(t)f_i(t)+ \sum_{i=1}^kh_i(1)\widetilde{f_i}(t)$.
Since $\sum_{i=1}^k h_i(1)f_i(1)=0$, 
we have 
$\sum_{i=1}^kh_i(1)\widetilde{f_i}(t)\in I$,
that is, $g(t)\in J+I$.
Hence we obtain $J+I\supset \ker (\phi)$.
Obviously, $\phi$ is a surjection.
By the homomorphism theorem, 
$\widetilde{\phi} : \mathbb{Z}[t^{\pm 1}]/(J+I) \to (1-t)M$ 
is a $\mathbb{Z}[t^{\pm 1}]$-isomorphism,
which is 
an isomorphism as quandles.
By Lemma \ref{conn. comp. iso.}, 
it follows that 
$\orb(j) \cong  \mathbb{Z}[t^{\pm 1}]/((f_1(t),f_2(t),\dots ,f_k(t))+I)$
for any $j=0,1,\dots ,a-1$.
\end{enumerate}
\end{proof}

By Theorem \ref{Alex. ccd}, 
we obtain the following corollaries, 
where we note that 
the dihedral quandle $R_m$ is isomorphic to 
the Alexander quandle $\mathbb{Z}[t^{\pm 1}]/(m,t+1)$ 
for any $m \in \mathbb{Z}_{\geq 0}$.

\begin{corollary}\label{conn. dihedral}
For any $m \in \mathbb{Z}_{>0}$, 
$R_m$ is a connected dihedral quandle 
if and only if 
$m$ is an odd number.
Furthermore, 
when $m$ is an even number, 
$R_m=\orb(0) \sqcup \orb(1)=\{0,2,\ldots ,m-2 \} \sqcup \{1,3,\ldots ,m-1\}$ is 
the connected component decomposition of $R_m$.
\end{corollary}

\begin{corollary}\label{conn. Alex. qnd.}
Let $f_1(t),f_2(t), \dots ,f_k(t) \in \mathbb{Z}[t^{\pm 1}]$.
Then the Alexander quandle 
$\mathbb{Z}[t^{\pm 1}]/(f_1(t),f_2(t),\dots ,f_k(t))$ 
is connected 
if and only if 
$\gcd (f_1(1),f_2(1), \dots ,f_k(1))=1$. 
\end{corollary}

For example, 
the tetrahedral quandle ${\mathbb Z}[t^{\pm 1}]/(2,t^2+t+1)$ is connected 
by Corollary \ref{conn. Alex. qnd.}.

\begin{corollary}\label{conn. comp. iso. of Alex. qnd. (m,t+a)}
Let $a,m \in \mathbb{Z}$ and let $n=m/\gcd (m, 1+a)$.
Then any connected component of the Alexander quandle $\mathbb{Z}[t^{\pm 1}]/(m,t+a)$
 is isomorphic to $\mathbb{Z}[t^{\pm 1}]/(n,t+a)$.
\end{corollary}

\begin{proof}
By Theorem \ref{Alex. ccd},
any connected component of $\mathbb{Z}[t^{\pm 1}]/(m,t+a)$
is isomorphic to $\mathbb{Z}[t^{\pm 1}]/((m,t+a)+I)$, 
where $I=\{ -a_2 \, |\, a_1,a_2 \in \mathbb{Z}[t^{\pm 1}],~ a_1 m+a_2(1+a)=0 \}$.
For any $-x_2\in I$,
there exists $x_1\in \mathbb{Z}[t^{\pm 1}]$
such that $x_1 m+x_2(1+a)=0$.
Then we have $x_1 n+x_2(1+a)/\gcd(m,1+a)=0$.
Since $n$ and $(1+a)/\gcd(m,1+a)$ are relatively prime, $x_2$ is devisible by $n$.
Hence we obtain $I\subset n\mathbb{Z}[t^{\pm 1}]$.
On the other hand, 
for any $x_2= ns \in n\mathbb{Z}[t^{\pm 1}]$, 
there exists $x_1=-s(1+a)/\gcd(m,1+a) \in \mathbb{Z}[t^{\pm 1}]$ 
such that $x_1 m+x_2(1+a)=0$.
Hence $-x_2 \in I$, 
that is, $I\supset n\mathbb{Z}[t^{\pm 1}]$.
Therefore $I= n\mathbb{Z}[t^{\pm 1}]$.
Since $m$ is divisible by $n$,
we have 
$\mathbb{Z}[t^{\pm 1}]/((m,t+a)+I)=
\mathbb{Z}[t^{\pm 1}]/(m,n,t+a)=\mathbb{Z}[t^{\pm 1}]/(n,t+a)$.
\end{proof}

\begin{corollary}\label{conn. comp. iso. of dihedral qnd.}
If $m$ is an even number, 
then any connected component of $R_m$ is isomorphic to $R_{m/2}$.
\end{corollary}

\section{The maximal connected subquandle decomposition}

In this section, we consider a decomposition of a quandle into 
the disjoint union of maximal connected subquandles, 
and 
show that 
it is uniquely obtained by iterating a connected component decomposition 
when the quandle is finite.
We remark that 
$\{ (1,2,3), (1,3,2) \}$ is a connected component of $\conj (S_3)$, 
but not a connected subquandle of it, 
where $S_3$ is a symmetric group of degree 3.

Let $X$ be a quandle and let $A$ be a connected subquandle of $X$.
We say that $A$ is a \emph{maximal connected subquandle} of $X$ 
when any connected subquandle of $X$ including $A$ is only $A$.  
We say that 
$X=\bigsqcup_{i \in I}A_i$ is 
the \emph{maximal connected subquandle decomposition} of $X$ 
when each $A_i$ is a maximal connected subquandle of $X$.

Let $X$ be a quandle and let $A$ be a subset of $X$.
For any $a, b \in X$, 
we write $a \sim_A b$ 
when there exists $f \in \inn (A)$ such that $f(a)=b$.
It is an equivalence relation on $X$.
It is easy to see that 
$a \sim_A b$ if and only if 
there exist $a_1,a_2, \ldots ,a_n \in A$ and $k_1,k_2, \ldots ,k_n \in \mathbb{Z}$ 
such that $a*^{k_1}a_1*^{k_2} \cdots *^{k_n}a_n=b$.
Furthermore $X$ is a connected quandle 
if and only if $a \sim_X b$ for any $a,b \in X$.

%

\begin{lemma}\label{conn. gen. by union}
Let $X$ be a quandle 
and let $A_i$ be  connected subquandles of $X$ for any $i \in I$. 
If $\bigcap_{i \in I} A_i \neq \emptyset$, 
then $\langle \bigcup_{i \in I}A_i \rangle$ is a connected subquandle of $X$.
\end{lemma}

\begin{proof}
Let $A=\bigcup_{i \in I} A_i$, 
and suppose that $\bigcap_{i \in I} A_i \neq \emptyset$.
For any $x,y \in \langle A \rangle$, 
there exist $a,b \in A$ 
such that $x \sim_A a$ and $y \sim_A b$.
For any $c \in \bigcap_{i \in I} A_i$, 
we have $a \sim_A c$ and $b \sim_A c$.
Hence we have 
$a \sim_A b$, which implies that $x \sim_A y$.
Therefore 
$x \sim_{\langle A \rangle} y$, that is, 
$\langle A \rangle$ is a connected subquandle of $X$.
\end{proof}

\begin{lemma}\label{empty or coincidence}
Let $X$ be a quandle 
and let $A_1$ and $A_2$ be maximal connected subquandles of $X$.
If $A_1 \cap A_2 \neq \emptyset$, 
then $A_1=A_2$.
\end{lemma}

\begin{proof}
Suppose that $A_1 \cap A_2 \neq \emptyset$.
By Lemma \ref{conn. gen. by union}, 
$\langle A_1 \cup A_2 \rangle$ is a connected subquandle of $X$.
Since $A_1$ and $A_2$ are included in $\langle A_1 \cup A_2 \rangle$ 
and maximal connected subquandles of $X$, 
we obtain $\langle A_1 \cup A_2 \rangle =A_1=A_2$.
\end{proof}

\begin{theorem}\label{uniquely cccd}
Any quandle has the unique maximal connected subquandle decomposition.
\end{theorem}

\begin{proof}
Let $X$ be a quandle.
For any $a \in X$, 
we define $[a]:=\bigcup_{a \in W} W$,  
where $W$ is a connected subquandle of $X$.
By Lemma \ref{conn. gen. by union}, 
$\langle [a] \rangle$ is a connected subquandle of $X$.
Suppose that $A$ is a connected subquandle of $X$ 
including $\langle [a] \rangle$.
Since $A$ contains $a$, 
we have $A \subset [a]$, that is, $A=[a]=\langle [a] \rangle$.
Hence $[a]$ is a maximal connected subquandle of $X$.
For any $a,b \in X$, 
we have $[a] \cap [b]= \emptyset$ 
or $[a]=[b]$ by Lemma \ref{empty or coincidence}. 
Therefore there exists a subset $Y$ of $X$ 
such that $X=\bigsqcup_{a \in Y}[a]$.
It is the maximal connected subquandle decomposition of $X$.
Next, we show the uniqueness.
Let $B$ be a maximal connected subquandle of $X=\bigsqcup_{a \in Y}[a]$.
Then there exists $a' \in Y$ 
such that $B \cap [a'] \neq \emptyset$.
By Lemma \ref{empty or coincidence}, 
we have $B=[a']$.
Therefore $X$ has the unique maximal connected subquandle decomposition.
\end{proof}

For a quandle $X$, 
it is easy to see that 
any connected subquandle of $X$ 
is included in some connected component of $X$.
Therefore 
if a connected component of $X$ is 
a connected subquandle of $X$, 
then it is a maximal connected subquandle of $X$.

%
%
%
%

Let $X$ be a quandle 
and let $\pqnd (X)$ be the set of all subquandles of $X$.
For any $\mathcal{A} \subset \pqnd(X)$, 
we define $D(\mathcal{A}):=\bigcup_{A \in \mathcal{A}} \{ \orb_{A}(a) \mid a \in A \}$.
It is easy to see that 
$\bigcup_{A \in \mathcal{A}}A=\bigcup_{A \in D(\mathcal{A})}A$. 
We put $D^0(\mathcal{A}):=\mathcal{A}$ and 
$D^{k+1}(\mathcal{A}) := D(D^k(\mathcal{A}))$ for any $k \in \mathbb{Z}_{\geq 0}$.

\begin{theorem}\label{cccd}
Let $X$ be a quandle.
If there exists $n \in \mathbb{Z}_{\geq 0}$ 
such that $D^n(\{X\})=D^{n+1}(\{X\})$,
then $X=\bigsqcup_{A \in D^n( \{ X \} )}A$ is 
the maximal connected subquandle decomposition of $X$.
In particular, 
if $X$ is a finite quandle, 
then there exists $n \in \mathbb{Z}_{\geq 0}$ 
such that $X=\bigsqcup_{A \in D^n(\{X\})}A$ is 
the maximal connected subquandle decomposition of $X$.
\end{theorem}

\begin{proof}
Suppose that there exists $n \in \mathbb{Z}_{\geq 0}$ 
such that $D^n(\{X\})=D^{n+1}(\{X\})$.
For any $A \in D^n(\{X\})$, 
$A$ has only one orbit $A$, 
that is, $A$ is a connected subquandle of $X$.
Let $Y$ be a connected subquandle of $X$ including $A$ 
and let $X_0=X$.
Then 
there exists an orbit of $X_0$ including $Y$. 
We denote by $X_1 \in D^1(\{X\})$ the orbit.
For any $i \in \mathbb{Z}_{\geq 0}$, 
we denote by $X_{i+1} \in D^{i+1}(\{X\})$ 
an orbit of $X_i$ including $Y$ inductively.
Then we have $A \subset Y \subset X_n$.
Since $A, X_n \in D^n(\{X\})$, 
we have $A=Y=X_n$.
Therefore $A$ is the maximal connected subquandle of $X$, 
and $X=\bigsqcup_{A \in D^n( \{ X \} )}A$ is 
the maximal connected subquandle decomposition of $X$.
Next, let $X$ be a finite quandle.
For any $i \in \mathbb{Z}_{\geq 0}$, 
we have $\# D^i(\{X\}) \leq \# D^{i+1}(\{X\}) \leq \#X$.
Since $\# X$ is finite, 
there exists $n \in \mathbb{Z}_{\geq 0}$ 
such that $\# D^n(\{X\}) = \# D^{n+1}(\{X\})$, 
which implies that $D^n(\{X\}) = D^{n+1}(\{X\})$.
Therefore $X=\bigsqcup_{A \in D^n( \{ X \} )}A$ is 
the maximal connected subquandle decomposition of $X$.
\end{proof}

By Lemma \ref{conn. comp. iso.} and Theorem \ref{cccd},
the following corollary holds immediately.

\begin{corollary}\label{Alex. cccd}
All maximal connected subquandles of 
a finite Alexander quandle 
are isomorphic.
\end{corollary}

For a quandle $X$, 
we denote by $\depth(X)$ the minimal number of $n$ 
satisfying $X=\bigsqcup_{A \in D^n(\{X\})}A$ is 
the maximal connected subquandle decomposition of $X$.
It is called the \emph{subquandle depth} of $X$ in \cite{NW06}.
Obviously, $X$ is a connected quandle 
if and only if $\depth(X)=0$.

\section{Examples of the maximal connected subquandle decomposition}

In this section, 
we give examples 
of the maximal connected subquandle decompositions of some quandles.

Let $S_n$ be a symmetric group of degree $n$.
We consider connectivity of $\conj(S_n)$.
By Example \ref{conjugation}, 
a connected component of $\conj (S_n)$ coincides with one of conjugacy classes of $S_n$.
We denote by $C(a)$ the conjugacy class of $S_n$ containing $a$.
We note that 
two elements of $S_n$ are conjugate if and only if their cyclic types coincide.

\begin{example}
\begin{enumerate}
\renewcommand{\labelenumi}{(\arabic{enumi})}
\item 
We show that the maximal connected subquandle decomposition of $S_3$ is 
\begin{eqnarray*}
S_3=\{ (1\, 2\, 3)\}\sqcup \{ (1\, 3\, 2)\}\sqcup C((1\, 2))\sqcup C(e).
\end{eqnarray*}
 $S_3=C((1\, 2\, 3))\sqcup C((1\, 2)) \sqcup C(e)$
is the connected component decomposition of $S_3$.
Furthermore, 
$C((1\, 2))$ and $C(e)$ are connected quandles, 
and 
$C((1\,2\,3))=\{ (1\, 2\, 3)\}\sqcup \{ (1\, 3\, 2)\}$ 
is the connected component decomposition of $C((1\,2\,3))$.
Therefore 
\begin{eqnarray*}
S_3=\{ (1\, 2\, 3)\}\sqcup \{ (1\, 3\, 2)\}\sqcup C((1\, 2))\sqcup C(e)
\end{eqnarray*}
is the maximal connected subquandle decomposition of $S_3$,
and we have $\depth(S_3)=2$.
\item
We show that the maximal connected subquandle decomposition of $S_4$ is 
\begin{eqnarray*}
S_4&=& C((1\,2\,3\,4))\\
&& \sqcup \{ (1\, 2\, 3),(1\, 4\, 2),(1\, 3\, 4),(2\, 4\, 3)\}\sqcup \{ (1\, 3\, 2),(1\, 2\, 4),(1\, 4\, 3),(2\, 3\, 4)\} \\
&& \sqcup \{ (1\, 2)(3\, 4)\} \sqcup \{(1\, 3)(2\, 4)\} \sqcup \{(1\, 4)(2\, 3)\} 
\sqcup C((1\,2)) \sqcup C(e).
\end{eqnarray*}
$S_4=C((1\,2\,3\,4))\sqcup C((1\,2\,3))\sqcup C((1\,2)(3\,4))\sqcup C((1\,2)) \sqcup C(e)$ 
is the connected component decomposition of $S_4$.
Furthermore, $C((1\,2\,3\,4))$, $C((1\,2))$ and $C(e)$ are connected quandles,
and $C((1\,2\,3))=\{ (1\, 2\, 3),(1\, 4\, 2),(1\, 3\, 4),(2\, 4\, 3)\}\sqcup \{ (1\, 3\, 2),(1\, 2\, 4),(1\, 4\, 3),(2\, 3\, 4)\}$
and $C((1\,2)(3\,4))=\{ (1\, 2)(3\, 4)\} \sqcup \{(1\, 3)(2\, 4)\} \sqcup \{(1\, 4)(2\, 3)\}$ are
 the connected component decompositions of $C((1\,2\,3))$ and $C((1\,2)(3\,4))$ respectively. 
Since any connected component of $C((1\,2\,3))$ and $C((1\,2)(3\,4))$ is connected,
\begin{eqnarray*}
S_4&=& C((1\,2\,3\,4))\\
&& \sqcup \{ (1\, 2\, 3),(1\, 4\, 2),(1\, 3\, 4),(2\, 4\, 3)\}\sqcup \{ (1\, 3\, 2),(1\, 2\, 4),(1\, 4\, 3),(2\, 3\, 4)\} \\
&& \sqcup \{ (1\, 2)(3\, 4)\} \sqcup \{(1\, 3)(2\, 4)\} \sqcup \{(1\, 4)(2\, 3)\} 
\sqcup C((1\,2)) \sqcup C(e)
\end{eqnarray*}
is the maximal connected subquandle decomposition of $S_4$,
and we have $\depth (S_4)=2$.
\end{enumerate}
\end{example}

\begin{example}
We show that 
the maximal connected subquandle decomposition of 
the dihedral quandle $R_0$ is 
\begin{align*}
R_0=\bigsqcup_{i \in \mathbb{Z}} \{i\}.
\end{align*}
We note that 
$R_0$ is isomorphic to 
the Alexander quandle $\mathbb{Z}[t^{\pm1}]/(t+1)$.
By Theorem \ref{Alex. ccd}, 
the connected component decomposition of $R_0$ is 
\begin{align*}
R_0=\orb(0) \sqcup \orb(1)=
\{ i \mid i: \mathrm{even} \} \sqcup \{ i \mid i: \mathrm{odd} \}.
\end{align*}
Since 
each connected component is isomorphic to $R_0$ 
by Corollary \ref{conn. comp. iso. of dihedral qnd.}, 
we have 
$D^n(\{R_0\})=\{ \{ 2^nj+i \mid j \in \mathbb{Z} \} \mid i=0,1,\ldots, 2^n-1 \}$ 
for any $n \in \mathbb{Z}_{\geq 0}$ 
by iterating a connected component decomposition.
Hence for any $a,b \in R_0$, 
there exists $l \in \mathbb{Z}_{>0}$ such that 
$a$ and $b$ are in distinct elements of $D^l(\{R_0\})$. 
Since any connected subquandle is included in a connected component, 
any connected subquandle of $R_0$ is included in an element of $D^l(\{ R_0 \})$. 
Therefore 
$a$ and $b$ are in distinct maximal connected subquandles of $R_0$, 
which implies that
$R_0=\bigsqcup_{i \in \mathbb{Z}} \{i\}$
is the maximal connected subquandle decomposition of $R_0$, 
and $\depth (R_0)=\infty$.
\end{example}

\begin{example}
We consider the maximal connected subquandle decomposition 
of the Alexander quandle $\mathbb{Z}[t^{\pm 1}]/(6,t^2+t+1)$.
Since $\gcd(6,1^2+1+1)=3$, 
$\mathbb{Z}[t^{\pm 1}]/(6,t^2+t+1)=\orb(0)\sqcup\orb(1)\sqcup \orb(2)$ 
is the connected component decomposition of $\mathbb{Z}[t^{\pm 1}]/(6,t^2+t+1)$, 
and
\begin{eqnarray*}
\orb(0)
&=\{ [g(t)]\, |\, g(t) \in \mathbb{Z}[t^{\pm 1}],~ g(1)\equiv 0\mod 3\}\\
&=\left\{
\begin{array}{l}
0,3,3t,1+2t,1+5t,2+t,2+4t,\\
3+3t,4+2t,4+5t,5+t,5+4t
\end{array}
\right\} ,\\
\orb(1)
&=\{ [g(t)]\,|\, g(t) \in \mathbb{Z}[t^{\pm 1}],~ g(1)\equiv 1\mod 3\}\\
&=\left\{
\begin{array}{ll}
1,4,t,4t,1+3t,2+2t,2+5t,\\
3+t,3+4t,4+3t,5+2t,5+5t
\end{array}
\right\} ,\\
\orb(2)
&=\{ [g(t)]\, |\, g(t) \in \mathbb{Z}[t^{\pm 1}],~ g(1)\equiv 2\mod 3\}\\
&=\left\{
\begin{array}{ll}
2,5,2t,5t,1+t,1+4t,2+3t,\\
3+2t,3+5t,4+t,4+4t,5+3t
\end{array}
\right\}
\end{eqnarray*}
by Theorem \ref{Alex. ccd}

Next, 
by Theorem \ref{Alex. ccd},
for any $i=0,1,2$, 
$\orb(i)$ is isomorphic to $\mathbb{Z}[t^{\pm 1}]/((6,t^2+t+1)+I)$,
where 
\begin{align*}
I &= \{ -a_2(t+2)\, |\, a_1,a_2\in\mathbb{Z}[t^{\pm 1}],~6a_1+3a_2=0\}\\
&= \{ -a_2(t+2)\, |\, a_1,a_2\in\mathbb{Z}[t^{\pm 1}],~a_2=-2a_1\}\\
&=2(t+2)\mathbb{Z}[t^{\pm 1}],
\end{align*}
which implies that 
$\mathbb{Z}[t^{\pm 1}]/((6,t^2+t+1)+I)=\mathbb{Z}[t^{\pm 1}]/(6,2(t+2),t^2+t+1)$.
By Theorem \ref{Alex. ccd},
we obtain that 
\begin{eqnarray*}
&& \mathbb{Z}[t^{\pm 1}]/(6,2(t+2),t^2+t+1)
 = \{ 0,3,2+t,5+t \}
\sqcup \{ 1,4,t,3+t \}
\sqcup \{ 2,5,1+t,4+t \}
\end{eqnarray*}
is the connected component decomposition of $\mathbb{Z}[t^{\pm 1}]/(6,2(t+2),t^2+t+1)$.
By the proof of Theorem \ref{Alex. ccd}, 
the map $\widetilde{\phi} : \mathbb{Z}[t^{\pm 1}]/(6,2(t+2),t^2+t+1)\to (1-t)(\mathbb{Z}[t^{\pm 1}]/(6,t^2+t+1))$ 
defined by $\widetilde{\phi}(x)=(1-t)x$ 
is an isomorphism.
Hence, by the proof of Lemma \ref{conn. comp. iso.},
\begin{eqnarray*}
\orb(0) &=& (1-t)(\mathbb{Z}[t^{\pm 1}]/(6,t^2+t+1))\\
&=& \{ 0, 3, 3t, 3+3t\}
\sqcup \{ 1+5t, 1+2t, 4+2t, 4+5t \}
\sqcup \{ 2+4t, 2+t, 5+t, 5+4t \} 
\end{eqnarray*}
is the connected component decomposition of $\orb (0)$.
Furthermore, 
for any $i=1,2$, 
the map $\phi_i : \orb (0) \to \orb (i)$ 
defined by $\phi_i (x)=x+i$ 
is an isomorphism 
by the proof of Lemma \ref{conn. comp. iso.}.
Therefore 
\begin{eqnarray*}
\orb(1)= \{ 1, 4, 1+3t, 4+3t\}
\sqcup \{ 2+5t, 2+2t, 5+2t, 5+5t \}
\sqcup \{ 3+4t, 3+t, t, 4t \}
\end{eqnarray*}
and
\begin{eqnarray*}
\orb(2)= \{ 2, 5, 2+3t, 5+3t\}
\sqcup \{ 3+5t, 3+2t, 2t, 5t \}
\sqcup \{ 4+4t, 4+t, 1+t, 1+4t \} 
\end{eqnarray*}
are the connected component decompositions of 
$\orb (1)$ and $\orb (2)$ respectively.

Finally, 
by Theorem \ref{Alex. ccd}, 
any connected component of $\mathbb{Z}[t^{\pm 1}]/(6,2(t+2),t^2+t+1)$
is isomorphic to $\mathbb{Z}[t^{\pm 1}]/((6,2(t+2),t^2+t+1)+I')$,
where 
\begin{align*}
I' &=\{ -2a_2-a_3(t+2)\, |\, a_1,a_2,a_3\in\mathbb{Z}[t^{\pm 1}],~6a_1+6a_2+3a_3=0\}\\
&=\{ -2a_2-a_3(t+2)\, |\, a_1,a_2,a_3\in\mathbb{Z}[t^{\pm 1}],~a_3=-2(a_1+a_2)\}\\
&=\{ -2a_2+2(a_1+a_2)(t+2) \mid a_1,a_2 \in\mathbb{Z}[t^{\pm 1}]\}\\
&=2\mathbb{Z}[t^{\pm 1}], 
\end{align*}
which implies that 
$\mathbb{Z}[t^{\pm 1}]/((6,2(t+2),t^2+t+1)+I')=\mathbb{Z}[t^{\pm 1}]/(2,t^2+t+1)$.
By Corollary \ref{conn. Alex. qnd.}, 
$\mathbb{Z}[t^{\pm 1}]/(2,t^2+t+1)$ is a connected quandle.
Therefore 
\begin{eqnarray*}
&&\mathbb{Z}[t^{\pm 1}]/(6,t^2+t+1) \\
&&= \{ 0, 3, 3t, 3+3t\}
\sqcup \{ 1+5t, 1+2t, 4+2t, 4+5t \}
\sqcup \{ 2+4t, 2+t, 5+t, 5+4t \} \\
&& \quad \sqcup  \{ 1, 4, 1+3t, 4+3t\}
\sqcup \{ 2+5t, 2+2t, 5+2t, 5+5t \}
\sqcup \{ 3+4t, 3+t, t, 4t \} \\
&& \quad \sqcup  \{ 2, 5, 2+3t, 5+3t\}
\sqcup \{ 3+5t, 3+2t, 2t, 5t \}
\sqcup \{ 4+4t, 4+t, 1+t, 1+4t \}
\end{eqnarray*}
is the maximal connected subquandle decomposition of $\mathbb{Z}[t^{\pm 1}]/(6,t^2+t+1)$, 
and we obtain that $\depth (\mathbb{Z}[t^{\pm 1}]/(6,t^2+t+1))=2$.
\end{example}

\begin{proposition}\label{Alex. cccd}
Let $n_0 \in \mathbb{Z}_{>0}$, 
$a\in \mathbb{Z}$ 
and put 
$n_{i+1}:=n_i / \gcd(n_i,1+a)$ for any $i \in \mathbb{Z}_{\geq 0}$.
Let $l$ be the minimal number satisfying $n_l=n_{l+1}$.
Then
the Alexander quandle $\mathbb{Z}[t^{\pm 1}]/(n_0,t+a)$ is decomposed into 
$N$ maximal connected subquandles, 
where $N=\prod_{i=0}^{l-1} \gcd(n_i,1+a)$, and 
any maximal connected subquandle of $\mathbb{Z}[t^{\pm 1}]/(n_0,t+a)$ 
is isomorphic to $\mathbb{Z}[t^{\pm 1}]/(n_l,t+a)$.
\end{proposition}

\begin{proof}
By Theorem \ref{Alex. ccd}, 
for any $i\in \mathbb{Z}_{i\geq 0}$, 
$\mathbb{Z}[t^{\pm 1}]/(n_i,t+a)$ is decomposed into 
$\gcd (n_i,1+a)$ maximal connected subquandles, 
and any maximal connected subquandle of $\mathbb{Z}[t^{\pm 1}]/(n_i,t+a)$ 
is isomorphic to $\mathbb{Z}[t^{\pm 1}]/(n_{i+1},t+a)$.
Hence for any $i\in\mathbb{Z}_{>0}$, 
any element of $D^i(\{ \mathbb{Z}[t^{\pm 1}]/(n_0,t+a)\})$ 
is isomorphic to $\mathbb{Z}[t^{\pm 1}]/(n_{i},t+a)$,
and we have 
$\#D^i(\{ \mathbb{Z}[t^{\pm 1}]/(n_0,t+a)\})=\prod_{j=0}^{i-1} \gcd(n_j,1+a)$.
By Corollary \ref{conn. Alex. qnd.},
$n_k=n_{k+1}$, that is, $\gcd (n_k, 1+a)=1$ 
if and only if 
$D^k(\{ \mathbb{Z}[t^{\pm 1}]/(n_0,t+a) \})=D^{k+1}(\{ \mathbb{Z}[t^{\pm 1}]/(n_0,t+a) \})$.
Hence $l$ is the minimal number satisfying $n_l=n_{l+1}$
if and only if $\depth(\mathbb{Z}[t^{\pm 1}]/(n_0,t+a))=l$.
By Theorem \ref{cccd},
$\mathbb{Z}[t^{\pm 1}]/(n_0,t+a)=\bigsqcup_{C \in D^l(\{ \mathbb{Z}[t^{\pm 1}]/(n_0,t+a)\})} C$ is 
the maximal connected subquandle decomposition of $\mathbb{Z}[t^{\pm 1}]/(n_0,t+a)$.
Since $N=\prod_{i=0}^{l-1} \gcd(n_i,1+a)=\#D^l(\{ \mathbb{Z}[t^{\pm 1}]/(n_0,t+a)\})$,
$\mathbb{Z}[t^{\pm 1}]/(n_0,t+a)$ is decomposed into 
$N$ maximal connected subquandles, and 
any maximal connected subquandle of $\mathbb{Z}[t^{\pm 1}]/(n_0,t+a)$ 
is isomorphic to $\mathbb{Z}[t^{\pm 1}]/(n_l,t+a)$.
\end{proof}

In Proposition \ref{Alex. cccd},  
if $1+a$ is a prime number, 
$\mathbb{Z}[t^{\pm 1}]/(n_0,t+a)$ is decomposed into 
$|1+a|^l$ maximal connected subquandles, and 
any maximal connected subquandle of $\mathbb{Z}[t^{\pm 1}]/(n_0,t+a)$ 
is isomorphic to $\mathbb{Z}[t^{\pm 1}]/(k,t+a)$, 
where $n_0=k(1+a)^l$
such that $k$ and $1+a$ are relatively prime integers.

By Proposition \ref{Alex. cccd}, 
the following corollary holds.

\begin{corollary}\label{dihedral cccd}
For any $m\in\mathbb{Z}_{>0}$, 
the dihedral quandle $R_m$ is decomposed into $2^l$ maximal connected subquandles, 
and any maximal connected subquandle of $R_m$ is 
isomorphic to $R_k$, and $\depth (R_m)=l$,
where $k$ is an odd number, and $l\in\mathbb{Z}_{>0}$ 
such that  $m=2^lk$.
\end{corollary}

\section{A multiple conjugation quandle and a $G$-family of quandles}

We recall the definition of a multiple conjugation quandle \cite{Ishii15}.

\begin{definition}
A \emph{multiple conjugation quandle} $X$ 
is a disjoint union of groups $G_{\lambda} (\lambda \in \Lambda)$ 
with a binary operation $* : X \times X \to X$ 
satisfying the following axioms.
\begin{itemize}
\item
For any $a,b \in G_\lambda$, 
$a*b = b^{-1}ab$.
\item
For any $x \in X$ and $a,b \in G_\lambda$, 
$x*e_\lambda = x$ and $x*(ab)=(x*a)*b$, 
where $e_\lambda$ is the identity of $G_\lambda$.
\item
For any $x,y,z \in X$, 
$(x*y)*z=(x*z)*(y*z)$.
\item
For any $x \in X$ and $a,b \in G_\lambda$, 
$(ab)*x=(a*x)(b*x)$, 
where $a*x, b*x \in G \mu$ for some $\mu \in \Lambda$.
\end{itemize}
\end{definition}

Let  $X=\bigsqcup_{\lambda \in \Lambda} G_\lambda$ 
be a multiple conjugation quandle.
In this paper, 
we write $a_1*a_2* \cdots *a_n$ for $(\cdots ((a_1*a_2*)*a_3)* \cdots *a_n)$ simply 
and 
denote by $G_a$ the group $G_\lambda$ containing $a \in X$.
we also denote by $e_\lambda$ the identity of $G_\lambda$.
Then the identity of $G_a$ is denoted by $e_a$ for any $a \in X$.
For any $a,b \in X$, 
we define a map $S_a : X \to X$ by $S_a(x)= x*a$ 
and 
a binary operation $*^{-1} : X \times X \to X$ by $a *^{-1}b=S_b^{-1}(a)$.

Let $X=\bigsqcup_{\lambda \in \Lambda}G_\lambda, Y=\bigsqcup_{\mu \in M}G_\mu$ and 
$(X,*_X), (Y,*_Y)$ be multiple conjugation quandles.
A \emph{homomorphism} $\phi : X \to Y$ is a map 
from $X$ to $Y$ satisfying $\phi(x*_Xy)=\phi(x)*_Y\phi(y)$ for any $x,y \in X$ 
and $\phi(ab)=\phi(a)\phi(b)$ for any $\lambda \in \Lambda$ and $a,b \in G_\lambda$.
We call a bijective homomorphism an \emph{isomorphism}.
$X$ and $Y$ are \emph{isomorphic},
denoted by $X \cong Y$, 
if there exists an isomorphism from $X$ to $Y$.
We call an isomorphism from $X$ to $X$ an \emph{automorphism} of $X$.
For any $a \in X$, 
the map $S_a$ is an automorphism of $X$.

Let $(X,*)$ be a multiple conjugation quandle.
A non-empty subset $Y$ of $X$ is called 
a \emph{sub-multiple conjugation quandle} of $X$ 
if $Y$ itself is a multiple conjugation quandle 
under $*$ and the group operations of $X$.

\begin{proposition}\label{subMCQ}
Let $X=\bigsqcup_{\lambda \in \Lambda}G_\lambda$ be a multiple conjugation quandle 
and let $Y$ be a non-empty subset of $X$.
Then the following are equivalent.
\begin{enumerate}
\renewcommand{\labelenumi}{(\arabic{enumi})}
\item
$Y$ is a sub-multiple conjugation quandle of $X$.
\item
For any $a,b \in Y$, $a*b \in Y$, 
and  $Y \cap G_\lambda$ is a subgroup of $G_\lambda$ or empty set 
for any $\lambda \in \Lambda$.
\item
For any $a,b \in Y$, $a*b \in Y$, 
and there exists a subset $\Lambda' \subset \Lambda$ such that 
$Y=\bigsqcup_{\lambda \in \Lambda'}H_\lambda$, 
where $H_\lambda$ is a subgroup of $G_\lambda$ for any $\lambda \in \Lambda'$.
\end{enumerate}
\end{proposition}

\begin{proof}
First suppose that it satisfies (1).
Then $a*b \in Y$ for any $a,b \in Y$ 
and $Y=\bigsqcup_{\mu \in \Lambda'}H_\mu$, 
where $H_\mu$ is a group for any $\mu \in \Lambda'$.
For any $\mu \in \Lambda'$, 
there uniquely exists $\lambda \in \Lambda$ 
such that $H_\mu$ is a subgroup of $G_\lambda$.
Then we again write $H_\lambda$ for $H_\mu$.
Since $H_\lambda \cap H_{\lambda'} = \emptyset$ $(\lambda \neq \lambda')$ 
for any $\lambda, \lambda' \in \Lambda'$, 
we have $Y \cap G_\lambda = H_\lambda$. 
Hence $H_\lambda$ is a subgroup of $G_\lambda$ for any $\lambda \in \Lambda'$, 
and we have $\Lambda' \subset \Lambda$.
Therefore it satisfies (3).
Second suppose that it satisfies (3).
Since $G_\lambda \cap G_{\lambda'}=\emptyset$ 
and  $H_\mu \cap H_{\mu'}=\emptyset$ 
when $\lambda \neq \lambda'$ and $\mu \neq \mu'$ 
for any $\lambda, \lambda' \in \Lambda$ and $\mu, \mu' \in \Lambda'$, 
we have $Y \cap G_\mu = H_\mu$ for any $\mu \in \Lambda'$ 
and $Y \cap G_\lambda = \emptyset$ for any $\lambda \in \Lambda - \Lambda'$.
Therefore it satisfies (2).
Finally suppose that it satisfies (2).
Let $\Lambda' := \{ \lambda \in \Lambda \mid Y \cap G_\lambda \neq \emptyset \}$.
Since $G_\lambda \cap G_{\lambda'}=\emptyset$ 
when $\lambda \neq \lambda'$ 
for any $\lambda, \lambda' \in \Lambda$, 
we have $Y=\bigsqcup_{\lambda \in \Lambda'}(Y \cap G_\lambda)$, 
where $Y \cap G_\lambda$ is a subgroup of $G_\lambda$ 
for any $\lambda \in \Lambda'$.
$Y$ clearly satisfies the axioms of a multiple conjugation quandle.
Therefore it satisfies (1).
\end{proof}

%

Let $X=\bigsqcup_{\lambda \in \Lambda}G_\lambda$ be a multiple conjugation quandle 
and let $A$ be a subset of $X$.
Then the minimal sub-multiple conjugation qundle of $X$ including $A$, 
denote by $\langle A \rangle_\mathsf{MCQ}$, 
is called the sub-multiple conjugation quandle generated by $A$.

We recall the definition of a $G$-family of quandles \cite{IIJO13}.

\begin{definition}
Let $G$ be a group with the identity element $e$.
A \emph{$G$-family of quandles} is a non-empty set $X$ 
with a family of binary operations $*^g:X \times X \to X (g \in G)$ 
satisfying the following axioms.
\begin{itemize}
\item
For any $x \in X$ and $g \in G$, 
$x *^g x=x$.
\item
For any $x,y \in X$ and $g,h \in G$, 
$x*^{gh}y=(x*^gy)*^hy$ and $x*^ey=x$.
\item
For any $x,y,z \in X$ and $g,h \in G$, 
$(x*^gy)*^hz=(x*^hz)*^{h^{-1}gh}(y*^hz)$.
\end{itemize}
\end{definition}

Then the following proposition holds.

\begin{proposition}[\cite{IIJO13}]
Let $(X,*)$ be a quandle and let $m$ be the type of $X$.
Then $(X,\{*^i\}_{i \in \mathbb{Z}_{km}})$ (resp. $(X,\{*^i\}_{i \in \mathbb{Z}})$) 
is a $\mathbb{Z}_{km}$ (resp. $\mathbb{Z}$)-family of quandles 
for any $k \in \mathbb{Z}_{>0}$.
\end{proposition}

Let $(X,\{ *^g \}_{g \in G})$ be a $G$-family of quandles.
Then $\bigsqcup_{x \in X} \{ x \} \times G$ is a multiple conjugation quandle 
with 
\[
(x,g)*(y,h):=(x*^hy,h^{-1}gh),
\hspace{10mm}
(x,g)(x,h):=(x,gh)
\]
for any $x,y \in X$ and $g,h \in G$.
We call it the \emph{associated multiple conjugation quandle} 
of  $(X,\{ *^g \}_{g \in G})$.

\section{The maximal connected sub-multiple conjugation quandle decomposition}

In this section, we consider a decomposition of a multiple conjugation quandle into 
the disjoint union of maximal connected sub-multiple conjugation quandles, 
and 
show that 
it is uniquely obtained by iterating a connected component decomposition 
when the multiple conjugation quandle is the finite disjoint union of groups.

Let $X=\bigsqcup_{\lambda \in \Lambda}G_\lambda$ 
which is a multiple conjugation quandle.
All automorphisms of $X$ form a group 
under composition of morphisms: $f \cdot g :=g \circ f$.
This group is called the \emph{automorphism group} of $X$ 
and denoted by $\aut (X)$.
For a subset $A$ of $X$, 
we denote by  $\inn (A)$ 
a subgroup of $\aut (X)$ generated by $\{ S_a \mid a \in A \}$.
In particular, $\inn (X)$ is called the \emph{inner automorphism group} of $X$. 
For any $\lambda \in \Lambda$ and $g \in \inn (A)$, 
we define an action of $\inn (A)$ on $\Lambda$ 
by $e_{\lambda \cdot g} = g(e_\lambda)$.
We say that $X$ is a \emph{connected multiple conjugation quandle} 
when the action is transitive.
In this paper, 
we call an orbit of $\Lambda$ by the action 
an orbit of $\Lambda$ simply 
and $\Lambda=\bigsqcup_{i \in I} \Lambda_i$ 
the orbit decomposition of $\Lambda$ 
when $\Lambda_i$ is an orbit of $\Lambda$ for any $i \in I$.
We denote by $\orb_\Lambda(\lambda)$ or $\orb(\lambda)$ 
the orbit of $\Lambda$ containing $\lambda$.
In general, for any orbit $\Lambda'$ of $\Lambda$, 
$\bigsqcup_{\lambda \in \Lambda'}G_\lambda$ is 
a sub-multiple conjugation quandle of $X$ 
and called a \emph{connected component} of $X$.
However $\bigsqcup_{\lambda \in \Lambda'}G_\lambda$ is not 
a connected sub-multiple conjugation quandle of $X$ in general.

\begin{proposition}\label{conn. comp. of quandle and MCQ}
Let $X$ be a quandle, $Y$ be a subquandle of $X$, $m$ be the type of $X$ 
and let $\bigsqcup_{x \in X} \{ x \} \times \mathbb{Z}_m$ be 
the associated multiple conjugation quandle 
of a $\mathbb{Z}_m$-family of quandles $(X,\{ *^i \}_{i \in \mathbb{Z}_m})$. 
Then $Y$ is a connected component of $X$ 
if and only if 
$\bigsqcup_{x \in Y} \{ x \} \times \mathbb{Z}_m$ is 
a connected component of $\bigsqcup_{x \in X} \{ x \} \times \mathbb{Z}_m$.
\end{proposition}

\begin{proof}
Suppose that 
for any $a,b \in Y$, 
there exists 
$f \in \inn (X)$ 
such that $f(a)=b$.
Then 
for any $(a,0),(b,0) \in \bigsqcup_{x \in Y} \{ x \} \times \mathbb{Z}_m$, 
there exists 
$f=S_{x_n}^{k_n} \circ \cdots \circ S_{x_2}^{k_2} \circ S_{x_1}^{k_1} \in \inn (X)$ 
such that $f(a)=b$.
Hence there exists 
$g=S_{(x_n,k_n)} \circ \cdots \circ S_{(x_2,k_2)} \circ S_{(x_1,k_1)} 
\in \inn (\bigsqcup_{x \in X} \{ x \} \times \mathbb{Z}_m)$ 
such that $g((a,0))=(b,0)$.
On the other hand, 
suppose that 
for any $a,b \in Y$, 
there exists 
$g \in \inn (\bigsqcup_{x \in X} \{ x \} \times \mathbb{Z}_m)$ 
such that $g((a,0))=(b,0)$.
Then 
for any $a,b \in Y$, 
there exists 
$g=S_{(x_n,k_n)}^{i_n} \circ \cdots \circ S_{(x_2,k_2)}^{i_2} \circ S_{(x_1,k_1)}^{i_1} 
\in \inn (\bigsqcup_{x \in X} \{ x \} \times \mathbb{Z}_m)$ 
such that $g((a,0))=(b,0)$.
Hence there exists 
$f=S_{x_n}^{i_nk_n} \circ \cdots \circ S_{x_2}^{i_2k_2} \circ S_{x_1}^{i_1k_1} \in \inn (X)$ 
such that $f(a)=b$.
Therefore  $Y$ is a connected component of $X$ 
if and only if 
$\bigsqcup_{x \in Y} \{ x \} \times \mathbb{Z}_m$ is a connected component 
of $\bigsqcup_{x \in X} \{ x \} \times \mathbb{Z}_m$.
\end{proof}

By Proposition \ref{conn. comp. of quandle and MCQ}, 
we obtain the following corollary immediately.

\begin{corollary}\label{conn. quandle and conn. MCQ}
Let $X$ be a quandle, $m$ be the type of $X$ 
and let $\bigsqcup_{x \in X} \{ x \} \times \mathbb{Z}_m$ be 
the associated multiple conjugation quandle 
of a $\mathbb{Z}_m$-family of quandles $(X,\{ *^i \}_{i \in \mathbb{Z}_m})$. 
Then $X$ is a connected quandle 
if and only if 
$\bigsqcup_{x \in X} \{ x \} \times \mathbb{Z}_m$ is 
a connected multiple conjugation quandle.
\end{corollary}

Let $X$ 
be a multiple conjugation quandle 
and let $A$ be a connected sub-multiple conjugation quandle of $X$.
We say that $A$ is a \emph{maximal connected sub-multiple conjugation quandle} of $X$ 
when any connected  sub-multiple conjugation quandle of $X$ 
including $A$ is only $A$.  
We say that 
$X=\bigsqcup_{i \in I}A_i$ is 
the \emph{maximal connected sub-multiple conjugation quandle decomposition} of $X$ 
when each $A_i$ is a maximal connected sub-multiple conjugation quandle of $X$.
By Corollary \ref{conn. quandle and conn. MCQ}, 
we obtain the following corollary.

\begin{corollary}\label{quandle ccc and MCQ ccc}
Let $X$ be a quandle, $m$ be the type of $X$ 
and let $\bigsqcup_{x \in X} \{ x \} \times \mathbb{Z}_m$ be 
the associated multiple conjugation quandle 
of a $\mathbb{Z}_m$-family of quandles $(X,\{ *^i \}_{i \in \mathbb{Z}_m})$.
Then the following hold.
\begin{enumerate}
\renewcommand{\labelenumi}{(\arabic{enumi})}
\item
$A$ is a maximal connected subquandle of $X$ 
if and only if 
$\bigsqcup_{x \in A} \{ x \} \times \mathbb{Z}_m$ is 
a maximal connected sub-multiple conjugation quandle of 
$\bigsqcup_{x \in X} \{ x \} \times \mathbb{Z}_m$.
\item
$X=\bigsqcup_{i \in I}A_i$ is 
the maximal connected subquandle decomposition of $X$
if and only if 
$\bigsqcup_{x \in X} \{ x \} \times \mathbb{Z}_m=
\bigsqcup_{i \in I}(\bigsqcup_{x \in A_i} \{ x \} \times \mathbb{Z}_m)$ is 
the maximal connected sub-multiple conjugation quandle decomposition 
of $\bigsqcup_{x \in X} \{ x \} \times \mathbb{Z}_m$.
\end{enumerate}
\end{corollary}

\begin{proof}
It is sufficient to prove (1) 
since (2) follows from (1) immediately.
Suppose that $A$ is a maximal connected subquandle of $X$.
Let $Y$ be a connected sub-multiple conjugation quandle 
including $\bigsqcup_{x \in A} \{ x \} \times \mathbb{Z}_m$.
Then there exists a connected subquandle $B$ of $X$ including $A$ 
such that $Y=\bigsqcup_{x \in B}\{x\} \times \mathbb{Z}_m$ 
by Corollary \ref{conn. quandle and conn. MCQ}.
Hence we have $A=B$ and $Y=\bigsqcup_{x \in A}\{x\} \times \mathbb{Z}_m$.
Therefore $\bigsqcup_{x \in A}\{x\} \times \mathbb{Z}_m$ is 
a maximal connected sub-multiple conjugation quandle 
of $\bigsqcup_{x \in X}\{x\} \times \mathbb{Z}_m$.
On the other hand, 
suppose that 
$\bigsqcup_{x \in A}\{x\} \times \mathbb{Z}_m$ is 
a maximal connected sub-multiple conjugation quandle 
of $\bigsqcup_{x \in X}\{x\} \times \mathbb{Z}_m$.
Let $B$ be a connected subquandle including $A$ 
and let $Y=\bigsqcup_{x \in B}\{x\} \times \mathbb{Z}_m$.
By Corollary \ref{conn. quandle and conn. MCQ}, 
$Y$ is a connected sub-multiple conjugation quandle 
of $\bigsqcup_{x \in X}\{x\} \times \mathbb{Z}_m$.
Hence we have $Y=\bigsqcup_{x \in A}\{x\} \times \mathbb{Z}_m$ and $A=B$.
Therefore $A$ is 
a maximal connected subquandle of $X$.
\end{proof}

\begin{example}
Let $m \in \mathbb{Z}_{>0}$.
Since the dihedral quandle $R_m$ is of type $2$, 
$\bigsqcup_{x \in R_m}\{x\} \times \mathbb{Z}_2$ is 
the associated multiple conjugation quandle 
of the $\mathbb{Z}_2$-family of quandles $(R_m,\{ *^i \}_{i \in \mathbb{Z}_2})$.
By Corollary \ref{dihedral cccd}, 
$R_m$ is decomposed into $2^l$ maximal connected subquandles, 
and any maximal connected subquandle of $R_m$ is 
isomorphic to $R_k$, 
where $k$ is an odd number, and $l\in\mathbb{Z}_{>0}$ 
such that  $m=2^lk$.
Therefore 
$\bigsqcup_{x \in R_m}\{x\} \times \mathbb{Z}_2$ 
is decomposed into $2^l$ maximal connected sub-multiple conjugation quandles, 
and any maximal connected sub-multiple conjugation quandles of 
$\bigsqcup_{x \in R_m}\{x\} \times \mathbb{Z}_2$ 
is isomorphic to $\bigsqcup_{x \in R_k}\{x\} \times \mathbb{Z}_2$
by Corollary \ref{quandle ccc and MCQ ccc}.
\end{example}

Let $X$ be a multiple conjugation quandle 
and let $A$ be a subset of $X$.
For any $a, b \in X$, 
we write $a \simmcq_A b$ 
when there exists $f \in \inn (A)$ such that $f(e_a)=e_b$.
It is an equivalence relation on $X$.
It is easy to see that 
$a \simmcq_A b$ 
if and only if 
there exist $a_1,a_2, \ldots ,a_n \in A$ 
and $k_1,k_2, \ldots ,k_n \in \mathbb{Z}$ 
such that $e_a*^{k_1}a_1*^{k_2} \cdots *^{k_n}a_n=e_b$.
Furthermore $X$ is a connected multiple conjugation quandle 
if and only if 
$a \simmcq_X b$ for any $a,b \in X$.

\begin{lemma}\label{generated MCQ}
Let $X$ be a multiple conjugation quandle 
and let $A$ be a subset of $X$.
Then for any $x \in \langle A \ranglemcq$, 
there exists $a \in A$ 
such that $x \simmcq_A a$.
\end{lemma}

\begin{proof}
Let $X=\bigsqcup_{\lambda \in \Lambda}G_\lambda$, 
$\Lambda'=\{ \lambda \in \Lambda \mid \langle A \rangle \cap G_\lambda \neq \emptyset \}$, 
where $\langle A \rangle$ is a subquandle of $X$ generated by $A$, 
and let $H_\lambda$ be a subgroup of $G_\lambda$ 
generated by $ \langle A \rangle \cap G_\lambda$ 
for any $\lambda \in \Lambda'$.
For any $x,y \in \bigsqcup_{\lambda \in \Lambda'}H_\lambda$, 
there exist $\lambda_1, \lambda_2 \in \Lambda'$ 
such that $x \in H_{\lambda_1}, y \in H_{\lambda_2}$, and
there exist $w_1,w_2,\ldots ,w_n \in \langle A \rangle \cap G_{\lambda_1}$, 
$v_1,v_2,\ldots ,v_m \in \langle A \rangle \cap G_{\lambda_2}$ and 
$k_1,k_2,\ldots ,k_n,l_1,l_2,\ldots ,l_m \in \mathbb{Z}$
such that $x=w_1^{k_1}w_2^{k_2}\cdots w_n^{k_n}$ and 
$y=v_1^{l_1}v_2^{l_2}\cdots v_m^{l_m}$.
Then we have 
$x*y=(w_1^{k_1}w_2^{k_2}\cdots w_n^{k_n})*(v_1^{l_1}v_2^{l_2}\cdots v_m^{l_m})
=(w_1*^{l_1}v_1*^{l_2}\cdots*^{l_m}v_m)^{k_1}
(w_2*^{l_1}v_1*^{l_2}\cdots*^{l_m}v_m)^{k_2}
\cdots (w_n*^{l_1}v_1*^{l_2}\cdots*^{l_m}v_m)^{k_n}$.
Since $\langle A \rangle$ is a quandle, 
$w_i*^{l_1}v_1*^{l_2}\cdots*^{l_m}v_m \in \langle A \rangle$ 
for any $i=1,2,\ldots ,n$, 
and then there exists $\lambda_0 \in \Lambda$ 
such that $w_i*^{l_1}v_1*^{l_2}\cdots*^{l_m}v_m \in G_{\lambda_0}$ 
for any $i=1,2,\ldots ,n$.
Since $\langle A \rangle \cap G_{\lambda_0} \neq \emptyset$, 
we have $\lambda_0 \in \Lambda'$ 
and $x*y \in H_{\lambda_0} \subset \bigsqcup_{\lambda \in \Lambda'}H_\lambda$.
Hence $\bigsqcup_{\lambda \in \Lambda'}H_\lambda$ is a 
sub-multiple conjugation quandle of $X$ by Proposition \ref{subMCQ}.
Then we have $\langle A \ranglemcq \subset \bigsqcup_{\lambda \in \Lambda'}H_\lambda$ 
since $A \subset \bigsqcup_{\lambda \in \Lambda'}H_\lambda$.
Hence for any $x \in \langle A \ranglemcq$, 
there exists $\mu \in \Lambda'$ 
such that $x \in H_\mu$ and $\langle A \rangle \cap G_\mu \neq \emptyset$.
Consequently, there exist $a,a_1,a_2,\ldots ,a_n \in A$ and $k_1,k_2,\ldots ,k_n \in \mathbb{Z}$ 
such that $a*^{k_1}a_1*^{k_2}\cdots*^{k_n}a_n \in \langle A \rangle \cap G_\mu$.
Therefore we have 
$e_x=(a*^{k_1}a_1*^{k_2}\cdots*^{k_n}a_n)(a*^{k_1}a_1*^{k_2}\cdots*^{k_n}a_n)^{-1}
=(aa^{-1})*^{k_1}a_1*^{k_2}\cdots*^{k_n}a_n
=e_a*^{k_1}a_1*^{k_2}\cdots*^{k_n}a_n$, 
which implies that 
$x \simmcq_A a$.
\end{proof}

%

\begin{lemma}\label{MCQ conn. gen. by union}
Let $X$ be a multiple conjugation quandle 
and let $A_i$ be connected sub-multiple conjugation quandles of $X$ 
for any $i \in I$. 
If $\bigcap_{i \in I} A_i \neq \emptyset$, 
then $\langle \bigcup_{i \in I}A_i \ranglemcq$ is 
a connected sub-multiple conjugation quandle of $X$.
\end{lemma}

\begin{proof}
Let $A=\bigcup_{i \in I} A_i$, 
and suppose that $\bigcap_{i \in I} A_i \neq \emptyset$.
For any $x,y \in \langle A \ranglemcq$, 
there exist $a,b \in A$ 
such that $x \simmcq_A a$ and $y \simmcq_A b$ 
by Lemma \ref{generated MCQ}.
For any $c \in \bigcap_{i \in I} A_i$, 
we have $a \simmcq_A c$ and $b \simmcq_A c$.
Hence we have 
$a \simmcq_A b$, which implies that $x \simmcq_A y$.
Therefore 
$x \simmcq_{\langle A \ranglemcq} y$, that is, 
$\langle A \ranglemcq$ is a connected sub-multiple conjugation quandle of $X$.
\end{proof}

\begin{lemma}\label{MCQ empty or coincidence}
Let $X$ be a multiple conjugation quandle 
and let $A_1$ and $A_2$ be maximal connected sub-multiple conjugation quandles of $X$.
If $A_1 \cap A_2 \neq \emptyset$, 
then $A_1=A_2$.
\end{lemma}

\begin{proof}
Suppose that $A_1 \cap A_2 \neq \emptyset$.
By Lemma \ref{MCQ conn. gen. by union}, 
$\langle A_1 \cup A_2 \ranglemcq$ is 
a connected sub-multiple conjugation quandle of $X$.
Since $A_1$ and $A_2$ are included in $\langle A_1 \cup A_2 \ranglemcq$ 
and maximal connected sub-multiple conjugation quandles of $X$, 
we obtain $\langle A_1 \cup A_2 \ranglemcq =A_1=A_2$.
\end{proof}

\begin{theorem}\label{uniquely cccd}
Any multiple conjugation quandle has 
the unique maximal connected sub-multiple conjugation quandle decomposition.
\end{theorem}

\begin{proof}
Let $X=\bigsqcup_{\lambda \in \Lambda}G_\lambda$ be a multiple conjugation quandle.
For any $a \in X$, 
we define $[a]:=\bigcup_{a \in W} W$,  
where $W$ is a connected sub-multiple conjugation quandle of $X$.
By Lemma \ref{MCQ conn. gen. by union}, 
$\langle [a] \ranglemcq$ is a connected sub-multiple conugation quandle of $X$.
Suppose that $A$ is a connected sub-multiple conjugation quandle of $X$ 
including $\langle [a] \ranglemcq$.
Since $A$ contains $a$, 
we have $A \subset [a]$, 
which implies that $A=[a]=\langle [a] \ranglemcq$.
Hence $[a]$ is a maximal connected sub-multiple conjugation quandle of $X$.
For any $a,b \in X$, 
we have $[a] \cap [b]= \emptyset$ or $[a]=[b]$ 
by Lemma \ref{MCQ empty or coincidence}. 
Therefore there exists a subset $Y$ of $X$ 
such that $X=\bigsqcup_{a \in Y}[a]$.
It is the maximal connected sub-multiple conjugation quandle decomposition of $X$.
Next, we show the uniqueness.
Let $B$ be a maximal connected sub-multiple conjugation quandle of $X=\bigsqcup_{a \in Y}[a]$.
Then there exists $a' \in Y$ such that $B \cap [a'] \neq \emptyset$.
By Lemma \ref{MCQ empty or coincidence}, 
we have $B=[a']$.
Therefore $X$ has the unique maximal connected sub-multiple conjugation quandle decomposition.
\end{proof}

For a multiple conjugation quandle $X$, 
it is easy to see that 
any connected sub-multiple conjugation quandle of $X$ 
is included in some connected component of $X$.
Therefore 
if a connected component of $X$ is 
a connected sub-multiple conjugation quandle of $X$, 
then it is a maximal connected sub-multiple conjugation quandle of $X$.

%
%
%
%

Let $X=\bigsqcup_{\lambda \in \Lambda}G_\lambda$ be 
a multiple conjugation quandle 
and let $\p(\Lambda)$ be the set 
of all subsets of $\Lambda$.
For any $\mathcal{A} \subset \p(\Lambda)$, 
we define 
$D(\mathcal{A}):=\bigcup_{\Lambda \in \mathcal{A}} \{ \orb_{\Lambda}(\mu) \mid \mu \in \Lambda \}$.
It is easy to see that 
$\bigcup_{\Lambda \in \mathcal{A}}\Lambda=\bigcup_{\Lambda \in D(\mathcal{A})}\Lambda$. 
We put $D^0(\mathcal{A}):=\mathcal{A}$ and 
$D^{k+1}(\mathcal{A}) := D(D^k(\mathcal{A}))$ for any $k \in \mathbb{Z}_{\geq 0}$.

\begin{theorem}\label{mcqcccd}
Let $X$ be a multiple conjugation quandle.
If there exists $n \in \mathbb{Z}_{\geq 0}$ 
such that $D^n(\{\Lambda \})=D^{n+1}(\{\Lambda \})$,
then $X=\bigsqcup_{\Lambda' \in D^n( \{ \Lambda \} )}
(\bigsqcup_{\lambda \in \Lambda'}G_\lambda)$ is 
the maximal connected sub-multiple conjugation quandle decomposition of $X$.
In particular, 
if $\Lambda$ is a finite set, 
then there exists $n \in \mathbb{Z}_{\geq 0}$ 
such that $X=\bigsqcup_{\Lambda' \in 
D^n( \{ \Lambda \} )}(\bigsqcup_{\lambda \in \Lambda'}G_\lambda)$ is 
the maximal connected sub-multiple conjugation quandle decomposition of $X$.
\end{theorem}

\begin{proof}
Suppose that there exists $n \in \mathbb{Z}_{\geq 0}$ 
such that $D^n(\{ \Lambda \})=D^{n+1}(\{ \Lambda \})$.
For any $\Lambda' \in D^n(\{ \Lambda \})$, 
$\Lambda'$ has only one orbit $\Lambda'$, 
that is, $\bigsqcup_{\lambda \in \Lambda'}G_\lambda$ is 
a connected sub-multiple conjugation quandle of $X$.
Let $Y=\bigsqcup_{\lambda \in M}H_\lambda$ be 
a connected sub-multiple conjugation quandle of $X$ 
including $\bigsqcup_{\lambda \in \Lambda'}G_\lambda$, 
where $H_\lambda$ is a subgroup of $G_\lambda$ 
for any $\lambda \in M$, 
and let $\Lambda_0=\Lambda$.
Then 
there exists an orbit of $\Lambda_0$ 
by the action of $\inn (\bigsqcup_{\lambda \in \Lambda_0}G_\lambda)$
including $M$.
We denote by $\Lambda_1 \in D^1(\{ \Lambda \})$ the orbit.
For any $i \in \mathbb{Z}_{\geq 0}$, 
we denote by $\Lambda_{i+1} \in D^{i+1}(\{\Lambda \})$ 
an orbit of $\Lambda_i$ 
by the action of $\inn (\bigsqcup_{\lambda \in \Lambda_i}G_\lambda)$
including $M$ inductively.
Then we have $\Lambda' \subset M \subset \Lambda_n$.
Since $\Lambda', \Lambda_n \in D^n(\{ \Lambda \})$, 
we have $\Lambda'=M=\Lambda_n$ 
and $Y=\bigsqcup_{\lambda \in \Lambda'}H_\lambda$.
Since $H_\lambda$ is a subgroup of $G_\lambda$ 
for any $\lambda \in \Lambda'$, 
we have $Y=\bigsqcup_{\lambda \in \Lambda'}G_\lambda$.
Therefore $\bigsqcup_{\lambda \in \Lambda'}G_\lambda$ is 
a maximal connected sub-multiple conjugation quandle of $X$, 
and  $X=\bigsqcup_{\Lambda' \in D^n( \{ \Lambda \} )}(\bigsqcup_{\lambda \in \Lambda'}G_\lambda)$ is 
the maximal connected sub-multiple conjugation quandle decomposition of $X$.
Next, let $\Lambda$ be a finite set.
For any $i \in \mathbb{Z}_{\geq 0}$, 
we have $\# D^i(\{\Lambda \}) \leq \# D^{i+1}(\{\Lambda \}) \leq \# \Lambda$.
Since $\# \Lambda$ is finite, 
there exists $n \in \mathbb{Z}_{\geq 0}$ 
such that $\# D^n(\{\Lambda\}) = \# D^{n+1}(\{\Lambda\})$, 
which implies that $D^n(\{\Lambda\}) = D^{n+1}(\{\Lambda\})$.
Therefore 
$X=\bigsqcup_{\Lambda' \in D^n( \{ \Lambda \} )}(\bigsqcup_{\lambda \in \Lambda'}G_\lambda)$ is 
the maximal connected sub-multiple conjugation quandle decomposition of $X$.
\end{proof}

\section*{Acknowledgment}
The authors would like to thank Atsushi Ishii 
for his helpful comments and suggestions.


\end{document}